\documentclass[12pt,leqno]{amsart}
\newtheorem{thm}{Theorem}[section]
\newtheorem{defi}{Definition}[section]

\newcommand{\be}{\begin{equation}}
\newcommand{\ee}{\end{equation}}
\numberwithin{equation}{section}
\newcommand{\bea}{\begin{eqnarray}}
\newcommand{\eea}{\end{eqnarray}}
\newcommand{\beb}{\begin{eqnarray*}}
\newcommand{\eeb}{\end{eqnarray*}}
\usepackage{amssymb,amsfonts,amsthm,setspace,indentfirst}
\usepackage[dvips]{graphics}
\usepackage{epsfig}

\begin{document}
\title{Rough Cauchy sequences in a cone metric space}
\author{Rahul Mondal}
\address{Department of Mathematics, Vivekananda Satavarshiki Mahavidyalaya, Manikpara, Jhargram -721513, West Bengal, India.} 

\email{imondalrahul@gmail.com, rahul04051987@gmail.com}

\begin{abstract}
Here we have introduced the idea of rough Cauchyness of sequences in a cone metric space. Also here we have discussed several basic properties of rough Cauchy sequences in a cone metric space using the idea of Phu \cite{PHU}.
\end{abstract}
\noindent\footnotetext{$\mathbf{2010}$\hspace{5pt}AMS\; Subject\; Classification: 40A05, 40A99.}
\maketitle
\section{\bf{Introduction}}
\noindent Metric spaces have been generalized in many ways. Huang and Zhang \cite{HLG} introduced the idea of cone metric spaces where the idea of distance have been generalized to a vector through the idea
of a cone defined in a ordered Banach space.\\ 
\indent Phu \cite{PHU} introduced the idea of rough convergence of sequences as a generalization of ordinary convergence of sequences in a normed linear space in 2001. There he also introduced the idea
of rough Cauchyness of sequences as a generalization of Cauchyness of sequences. Phu \cite{PHU} discussed about the usefulness of rough convergence and some of its basic properties. Phu \cite{PHU1} also
extended the idea of rough convergence of sequences in an infinite dimensional normed linear space in 2003 and also in 2008 Ayter \cite{AYTER1} introduced the  idea of rough statistical convergence. Recently Banerjee and Mondal \cite{RMROUGH} studied the idea of rough convergence of sequences in a cone metric space. Many works have been done by many authors \cite{AYTER2, PMROUGH1, PMROUGH2} using the idea of Phu. Here we have discussed the rough Cauchyness of sequences in a cone metric space.\\
\section{\bf{Preliminaries}}\label{preli}
\begin{defi}\cite{PHU}
Let $\{ x_{n} \}$ be a sequence in a normed linear space $(X, \left\| . \right\|)$,  and $r$ be a nonnegative real number. Then $\{ x_{n} \}$ is said to be $r$-convergent to $x$ if for any $\epsilon >0$, there exists a natural number $k$ such that $\left\|x_{n} - x \right\| < r +\epsilon$ for all $n \geq k$.
\end{defi}
\begin{defi} \cite{HLG}
Let $E$ be a real Banach space and $P$ be a subset of $E$. Then $P$ is called a cone if and only if $(i)$ $P$ is closed nonempty, and $P \neq \{0\}$\\
$(ii)$ $a,b \in \mathbb{R}$, $a,b \geq 0$, $x,y \in P$ implies $ax+by \in P$.\\
$(iii)$ $x \in P$ and $-x \in P$ implies $x=0$.
\end{defi}
Let $E$ be a real Banach space and $P$ be a cone in $E$. Let us use the partial ordering \cite{HLG} with respect to P by $x \leq y$ if and only if $y-x \in P$. We shall write $x <y$ to indicate that $x \leq y$ but $x \neq y$.\\
\indent Also by $x<< y$, we mean $y -x \in intP$, the interior of P. The cone P is called normal if there is a number $K >0$ such that for all $x,y \in E$, $0 \leq x \leq y$ implies $||x|| \leq K ||y||$. 
\begin{defi} \cite{HLG}
Let $X$ be a non empty set. If the mapping $d:X\times X \longrightarrow E$ satisfies the following three conditions\\ 
$(d1)$ $0 \leq d(x,y)$ for all $x,y \in X$ and $d(x,y)=0$ if and only if $x=y$;\\
$(d2)$ $d(x,y)=d(y,x)$ for all $x,y \in X$;\\
$(d3)$ $d(x,y) \leq d(x,z) + d(z,y)$ for all $x, y, z \in X$; then $d$ is called a cone metric on $X$, and $(X,d)$ is called a cone metric space.\\
\indent It is clear that a cone metric space is a generalization of metric spaces. Throughout $(X,d)$ or simply $X$ stands for a cone metric space which is associated with a real Banach space $E$ with a cone $P$, $\mathbb{R}$ for the set of all real numbers, $\mathbb{N}$ for the set of all natural numbers, sets are always subsets of $X$ unless otherwise stated.\end{defi}
\begin{defi} \cite{HLG}
Let $(X, d)$ be a cone metric space. A sequence $\{ x_{n} \}$ in $X$ is said to be convergent to $x\in X$ if for every $c \in E$ with $0<<c$ there is $k\in \mathbb{N}$ such that $d(x_{n}, x)<<c$, whenever for all $n>k$.
\end{defi}
We know that in a real Banach space $E$ with cone $P$. If $x_{0} \in intP$ and $c(> 0)\in \mathbb{R}$ then $cx_{0} \in intP$ and if $x_{0} \in P$ and $y_{0} \in intP$ then $x_{0} + y_{0} \in intP$. Hence we can also say that if $x_{0}, y_{0} \in intP$ then $x_{0} + y_{0} \in intP$. Also it has been discussed in \cite{RMROUGH} that a real normed linear space is always connected and if $E$ be a real Banach space with cone $P$ then $0 \notin intP$.
\begin{defi}\cite{PHU, RMROUGH} 
Let $(X,d)$ be a cone metric space. A sequence $\left\{x_{n} \right\}$ in $X$ is said to be $r$-convergent to $x$ for some $r \in E$ with $0<<r$ or $r=0$ if for every $\epsilon$ with $0<< \epsilon$ there exists a $k \in \mathbb{N}$ such that $d(x_{n}, x)<<r + \epsilon$ for all $n \geq k$.
\end{defi}
Let $(X,d)$ be a cone metric space with normal cone $P$ and normal constant $k$. Then from \cite{RMROUGH} we can say for every $\epsilon >0$, we can choose $c \in E$ with $c \in int P$ and $k\left\|c\right\| < \epsilon$ and also for each $c \in E$ with $0<<c$, there is a $\delta > 0$, such that $\left\|x\right\| < \delta $ implies $c- x \in int P$. We will use this idea in the next section of our work.
\section{\bf{$r$-Cauchy sequences in a cone metric space}}
\begin{defi}
Let $(X, d)$ be a cone metric space. A sequence $\{x_{n} \}$ in $X$ is said to be a cauchy sequence in $X$ if for every $(0 <<)\epsilon$ there exists a natural number $m$ such that $d(x_{i}, x_{j})<<
\epsilon$ for all $i,j \geq m$.
\end{defi}
\begin{defi}
 A sequence $\{x_{n} \}$ in a cone metric space $(X, d)$ is said to be a $r$-cauchy sequence for some$(0<<)r$ or $r=0$ in this space if for every $(0 <<)\epsilon$ there exists a natural number $m$ such that $d(x_{i}, x_{j})<< r + \epsilon$ fior all $i,j \geq m$.
\end{defi}
It should be noted that when $r=0$ then rough Cauchyness coincides with the Cauchyness.
\begin{thm}
Every $\frac{r}{2}$-convergent sequence in a cone metric space $(X, d)$ is $r$-cauchy for every $r$ as defined above.
\end{thm}
\begin{proof}

 Let $\{x_{n} \}$ be a $\frac{r}{2}$-convergent sequence in a cone metric space $(X, d)$ and converges to $x$ in $X$ and consider a arbitrary $\epsilon \in intP$. Now for the $(0 <<)\epsilon$ there exists a natural number $m$ such that $d(x_{n}, x)<< \frac{r}{2} +\frac{\epsilon}{2}$ for all $n \geq m$, so $[\frac{r}{2} +\frac{\epsilon}{2}] - d(x_{n}, x) \in intP$ for all $n \geq m$. Now if $r \in intP $ then for any two natural numbers $i, j$ we have $d(x_{i}, x_{j})\leq d(x_{i}, x)+ d(x_{j}, x)$ $\longrightarrow(i)$. So $[d(x_{i}, x)+ d(x_{j}, x)]- d(x_{i}, x_{j}) \in P$. Also for $i \geq m$ and $j \geq m$  we have $[\frac{r}{2} + \frac{\epsilon}{2}]-d(x_{i}, x) \in intP$ and $[\frac{r}{2} +\frac{\epsilon}{2}]-d(x_{j}, x) \in intP$. Hence $[r+\epsilon] -[d(x_{i}, x)+ d(x_{j}, x)] \in intP \longrightarrow (ii)$ Therefore using $(i)$ and $(ii)$ we have $d(x_{i}, x_{j}) << r+\epsilon$.
\end{proof}
The idea of boundedness of a sequence in cone metric spaces is similar as in metric spaces and this has been discussed in \cite{RMROUGH}. We recall that a sequence $\{x_{n} \}$ in a cone metric space $(X,d)$ is bounded if there is a $g \in intP$ such that $d(x_{m}, x_{n}) << g$ for all $m \in \mathbb{N}$ and $n \in \mathbb{N}$.\\ 
\begin{thm}
A bounded sequence in a cone metric space is always $r$-Cauchy for some $r$ as defined above.\end{thm}
\begin{proof}
Let $\{ x_{n} \}$ be a bounded sequence in a cone metric space $(X,d)$. So there exists a $(0<<)s$ such that $d(x_{n}, x_{m})<< s$ for all $m,n \in \mathbb{N}$. Therefore $[s - d(x_{n}, x_{m})] \in intP$
for all $m,n \in \mathbb{N}$. Hence for any $(0<<) \epsilon$ we have $[s - d(x_{n}, x_{m})]+ \epsilon \in intP$ for all $m,n \in \mathbb{N}$. Therefore $d(x_{n}, x_{m})<< s + \epsilon$ for all $m,n \in \mathbb{N}$. So $\{ x_{n} \}$ is a $s$-Cauchy sequence in $X$.
\end{proof}
\begin{thm}
Let $(X, d)$ be a cone metric space with normal cone $P$ and normal constant $k$ and with the given condition $*$ as given bellow. If $\{x_{n} \}$ and $\{y_{n} \}$ be two $\frac{r}{2k^{2}}$-Cauchy sequence $(0<<r)$ in $X$ then the sequence $\{ d(x_{n}, y_{n}) \}$ is $||r||$-Cauchy.
\end{thm}
$(*)$ If $p \leq c$ and $-p \leq c$ with $c \in P$ then $||p|| \leq k ||c||$, where $p$, $c \in E$.\\
\begin{proof}
Let $\epsilon >0$ be preassigned real number. So by using the property of a normal cone there exists a $c \in intP$ such that $k||c|| < \epsilon $. Since $c \in intP$, we have $\frac {c}{2} \in intP$ and also $\frac{r}{2k^{2}}\in intP$. Now there exists two positive integers $k_{1}$ and $k_{2}$ such that $d(x_{i}, x_{j})<< \frac{r}{2k^{2}} + \frac{c}{2}  $ for all $i,j \geq k_{1}$ and $d(y_{i}, y_{j})<< \frac{r}{2k^{2}} + \frac{c}{2} $ for all $i,j \geq k_{2}$. If $k$ be the maximum of $k_{1}$ and $k_{2}$ then by using the property of normal cone we have $||d(x_{i}, x_{j})||\leq k||\frac{r}{2k^{2}} + \frac{c}{2} || $ and $||d(y_{i},y_{j})||\leq k||\frac{r}{2k^{2}} + \frac{c}{2} || $ for all $i,j \geq k$.\\
\indent Now $d(x_{i},y_{i}) - d(x_{j},y_{j}) \leq d(x_{i}, x_{j}) + d(y_{i},y_{j})$ and $d(x_{j},y_{j}) - d(x_{i},y_{i}) \leq  d(x_{i}, x_{j}) + d(y_{i},y_{j})$. Since $[d(x_{i}, x_{j}) + d(y_{i},y_{j})] \in P$, by the given condition we have $||d(x_{i},y_{i})- d(x_{j},y_{j})|| \leq k ||d(x_{i}, x_{j}) + d(y_{i},y_{j})|| \leq k ||d(x_{i}, x_{j})|| +k|| d(y_{i},y_{j})||$. Hence for all $i, j \geq k$ we have $||d(x_{i},y_{i})- d(x_{j},y_{j})|| \leq ||r|| + k||c|| < ||r|| + \epsilon $.
\end{proof}
\begin{thm}
Let $\{ x_{n} \}$ and $\{ y_{n} \}$ be two sequences in a cone metric space $(X, d)$ such that $d( x_{n}, y_{n}) \longrightarrow 0$ as $n \longrightarrow \infty$. Then $\{ x_{n} \}$ is $r$-Cauchy if and only if $\{ y_{n} \}$ is $r$-Cauchy. 
\end{thm}
\begin{proof}
Let $\{ x_{n} \}$ be a $r$-Cauchy sequence in $(X, d)$. Then for $0(<<\epsilon)$ there exists two positive integers $k_{1}$ and $k_{2}$ such that $d(x_{i}, x_{j})<<r+ \frac{\epsilon}{3}$ for all $i,j \geq k_{1}$ and $d(x_{i}, y_{i})<< \frac{\epsilon}{3}$ for all $i \geq k_{2}$. Now $d(y_{i}, y_{j}) \leq d(x_{i}, y_{i}) + d(x_{i}, y_{j})$ and also $d(x_{i}, y_{j}) \leq d(x_{i}, x_{j}) + d(x_{j}, y_{j})$. Hence $d(x_{j}, y_{j}) \leq d(x_{i}, y_{i}) + d(x_{i}, x_{j}) + d(x_{j}, y_{j})$. If $k= max(k_{1}, k_{2})$, then for all $i,j \geq k$ we have $d(x_{i}, y_{j})<< r + \epsilon $ and hence $\{ y_{n} \}$ is $r$-Cauchy.\\ 
Conversely if $\{ y_{n} \}$ is $r$-Cauchy then similarly we can show that $\{ x_{n} \}$ is $r$-Cauchy. 
\end{proof}
\begin{thm}
Let $(X,d)$ be a cone metric space with normal cone $P$ and normal constant $k$. If a sequence $\left\{x_{n}\right\}$ in $(X,d)$ is $r$-Cauchy and also converges to $x$ in $X$ then the sequence $\left\{ d(x_{n}, x)-r\right\}$ is converges to 0 in $E$, provided that $\left\{ d(x_{n}, x)-r\right\}$ is a sequence in $P$.
\end{thm}

\begin{proof}
Let $\left\{x_{n}\right\}$ be a $r$-Cauchy sequence and converges to $x$ in $X$. Let $\epsilon >0$ be preassigned. Then we have an element $c \in E$ with $0<<c$ and $k \left\| c \right\| < \epsilon$. Now for $0<<c$ there exists  $k_{1}, k_{2} \in \mathbb{N}$ such that $d(x_{n},x_{m})<<r+\frac{c}{2}$ for all $n, m \geq k_{1}$ and $d(x_{m},x)<< \frac{c}{2}$ for all $m \geq k_{2}$. Hence $0 \leq d(x_{n},x)-r<<c$ for all $ n \geq k$, $k$ is the maximum of $k_{1}$ and $k_{2}$. Now by using the property of a normal cone we have $\left\| d(x_{n},x)-r \right\| \leq k \left\| c \right\| < \epsilon$ for all $n \geq k$. Therefore $\left\{ d(x_{n},x)-r\right\}$ converges to $0$ in $E$.
\end{proof}

\end{document}